\documentclass[12pt]{article}
\usepackage{graphicx}
\usepackage{amsfonts,amsmath,amssymb,amsthm, tikz}
\usepackage{fullpage,parskip}
\usepackage{authblk}
\usepackage{enumerate}

\newtheorem{theorem}{Theorem}[section]
\newtheorem{corollary}[theorem]{Corollary}
\newtheorem{lemma}[theorem]{Lemma}

\title{The Gamma-switch Ramsey Number}

\author[1]{Christopher Duffy}
\author[1]{Benjamin Fok}
\author[2]{Gary MacGillivray}
\affil[1]{School of Mathematics and Statistics, University of Melbourne, Australia}
\affil[2]{Department of  Mathematics and Statistics, University of Victoria, Canada}

\date{}
\begin{document}

\maketitle

\begin{abstract}
We define and develop preliminary theoretical results for the $\Gamma$-switch Ramsey number, a variation on the classical $m$-colour Ramsey number for which we allow permuting the colours incident with a vertex using elements of a group $\Gamma \leq S_m$.
We find bounds for the $\Gamma$-switch Ramsey number for groups with various properties as a function of the classical parameter.
We prove  $R_{C_3}(4,4,4) = R(3,3,3) + 1$, $R_{C_4}(3,4,3,4) = R(2,3,2,3) + 1$ and $43 \leq R_{C_4}(4,4,4,4) \leq R(3,3,3,3) + 1$.
\end{abstract}


\section{Introduction and Preliminaries}
Let $G$ be an $m$-edge-coloured complete graph.
Let $c(e) \in \{1,2,\dots,m\}$ denote the colour of edge $e$.
Let $\Gamma$ be a subgroup of $S_m$ and let $\pi \in \Gamma$.
For $v \in V(G)$, we \emph{switch at $v$ with $\pi$} by changing the colours of the edges incident with $v$ with respect to permutation $\pi$.
We denote by $G^{(\pi,v)}$ the resulting $m$-edge-coloured graph.
For all $1 \leq i \leq m$, an edge incident with $v$ of colour $i$ in $G$ has colour $\pi(i)$ in $G^{(\pi,v)}$.

We define 
\[G^{(\pi_1,v_1),(\pi_2,v_2)}= (G^{(\pi_1,v_1)})^{(\pi_2,v_2)},\]
and consider the notion of a \emph{switching sequence},
$\Pi = (\pi_1,v_1),(\pi_2,v_2),\dots, (\pi_k,v_k)$, which defines a sequence of switches at vertices of $G$.
The sequence $v_1,v_2,\dots, v_k$ is a sequence of not necessarily distinct vertices of $G$.
And the sequence $\pi_1, \pi_2,\dots, \pi_k$ is a sequence of not necessarily distinct elements of $\Gamma$.
We denote the resulting $m$-edge-coloured complete graph as $G^\Pi$.
That is 
\[
G^\Pi = G^{ (\pi_1,v_1),(\pi_2,v_2),\dots, (\pi_k,v_k)}.
\]
The $m$-edge-coloured graph $G^\Pi$ results from sequentially switching at $v_i$ with $\pi_i$ for $1 \leq i 
\leq k$.

For fixed $\Gamma \leq S_m$ the notion of $\Gamma$-switching defines an equivalence relation, $\sim$, on the set of all $m$-edge-coloured complete graphs on the same vertex set.
We say $G$ and $H$ are $\Gamma$-\emph{switch equivalent} when there exists a switching sequence $\Pi$ such that $G^\Pi = H$.
Let $\mathcal{K}^m_n$ be the set of all $m$-edge-coloured complete graphs on $n$ vertices with vertex set $\{u_1,u_2,\dots, u_n\}$.
For $G \in \mathcal{K}^m_n$, let  $[G]$ denote the equivalence class of $G$ with respect to this equivalence relation.

Let $\Gamma$ be a fixed subgroup of $S_m$.
In this work, we define the \emph{$\Gamma$-switch Ramsey number}.
Let $K_t^{(i)}$ denote a  copy of $K_t$ where all edges have colour $i$.
Let $a_1,a_2,\dots, a_m \geq 2$.
The parameter $R_{\Gamma}(a_1,a_2,\dots, a_m)$ is defined to be the least integer $n$ such that every
$m$-edge-coloured complete graph on $n$ vertices is $\Gamma$-switch equivalent to one that contains, for some $1 \leq i \leq m$, a copy of $K_{a_i}^{(i)}$.
In other words  $R_{\Gamma}(a_1,a_2,\dots, a_m)$ is the least integer $n$ such that every equivalence class of $\mathcal{K}^m_n/\sim$ contains a graph that contains, for some $1\leq i \leq m$, a copy of $K_{a_i}^{(i)}$.

Given the significant volume of work on the classical Ramsey numbers,
we refer the reader to the comprehensive dynamic survey on Ramsey numbers \cite{EJCDynamic} rather than provide further background here.
Though there has been no past work on the $\Gamma$-switch Ramsey numbers directly, in 2024 Mutar, Sivaraman and Slilaty provided the first results on the signed Ramsey number \cite{M24}. 
The signed Ramsey number is defined for the class of graphs where each edge is assigned a sign, $+$ or $-$.
The signed Ramsey number $R_{\pm}(a,b)$ is the least integer $n$ such that every complete signed graph on $n$ vertices is equivalent under the re-signing operation to one containing a $K_a$ with all positive edges or a $K_b$ with all negative edges.
By replacing positive and negative edges with red and blue edges (i.e., edges of colour $1$ and $2$), it follows directly that $R_{\pm}(a,b) = R_{S_2}(a,b)$.
In addition to the values given in Table \ref{tab:smallNumbers}, Mutar, Sivaraman and Slilaty prove $R_{S_2}(3,t) = t$ for all $t \geq 3$.
Our work herein on $\Gamma$-switching in $m$-edge-coloured graphs can be viewed as one possible generalisation of the signed Ramsey number.

\begin{table}[ht]
	\label{tab:smallNumbers}
	\begin{center}
		\begin{tabular}{c|c|c}
			& $R_{S_{2}}(\cdot,\cdot)$ & $R(\cdot,\cdot)$ \\
			\hline
			$(3,3)$ & $3$ & $6$\\
			$(3,4)$ & $4$ & $9$\\
			$(4,4)$ & $7$ & $18$\\
			$(4,5)$ & $8$ & $25$\\
		\end{tabular}
	\end{center}
	\caption{Small $S_2$-switch Ramsey Numbers.
    }
\end{table}

First introduced in the context of behavioural sciences by Abelson and Rosenberg \cite{A58} and developed by Zasavsky in the context of signed graphs \cite{zaslavsky1982}, switching on $2$-edge-coloured graphs was extended by Brewster and Graves \cite{B09} to switching in an $m$-edge-coloured graph with respect to a cyclic group as part of their consideration of homomorphisms under the equivalence relation defined by the switching operation.
This work was extended by Kidner in their masters thesis regarding switching operations and $2$-colourings of $(m,n)$-mixed graphs \cite{K21}.
In \cite{B25} the authors give a dichotomy theorem for $\Gamma$-switchable-colouring on $m$-edge-coloured graphs.

For graph and group theoretic notation not defined herein, we refer the reader respectively to \cite{bondy} and \cite{G07}.
We denote the set $\{1,2,\dots,m\}$ by $[m]$.

In Section \ref{sec:gen} we provide general bounds for a $\Gamma$-switch Ramsey number for various classes of groups, including abelian, semi-regular and groups for which the commutator subgroup is non-trivial.
In Section \ref{sec:trans} we focus our attention on groups that act transitively on the set of colours, giving particular attention to cyclic groups of odd and even order. 
Here we prove $R_{C_4}(3,4,3,4) = R(2,3,2,3) + 1$,  $R_{C_3}(4,4,4) = R(3,3,3) + 1$ and $43 \leq R_{C_4}(4,4,4,4) \leq R(3,3,3,3) + 1$.
This latter result provides a possible new approach to the long-standing open problem of improving the lower bound on $R(3,3,3,3)$.

\section{General Results}\label{sec:gen}

When $\Gamma$ acts intransitively on $[m]$, the set of colours,  we find an upper bound on the $\Gamma$-switch Ramsey number using the classical parameter where the colours correspond to the orbits of the action of $\Gamma$ on $[m]$.
Our result below follows from the following observation: if $G$ can be $\Gamma$-switched to contain a complete monochromatic subgraph, then all of the edges in that subgraph must have originally been in the same orbit of the action of $\Gamma$ on $[m]$.

\begin{theorem}\label{thm:inTrans}
	Let $\Gamma \leq S_m$.
	Let  $\{O_1,O_2,\dots, O_t\}$ be the set of orbits of the action of $\Gamma$ on $[m]$.
	Let $O_i = \{c_{1,i}, c_{2,i},\dots c_{|O_i|,i}\}$.
	For $n_i = R_{\Gamma|_{O_i}}(a_{c_{1,i}}, a_{c_{2,i}},\dots a_{c_{|O_i|,i}})$,
	\[
	R_\Gamma(a_1,a_2,\dots, a_m) \leq R(n_1,n_2,\dots,n_t).
	\]
\end{theorem}

\begin{proof}
	Let $n = R(n_1,n_2,\dots,n_t)$ and let $G$ be a complete $m$-edge-coloured graph on $n$ vertices.
	Since $n = R(n_1,n_2,\dots,n_t)$, there exists $1 \leq i \leq t$ such that $G$ contains a complete graph $H$ on $n_i$ vertices in which all of the edge colours are contained in $O_i$.
	The subgraph $H$ is a complete graph on $n_i$ vertices where each edge colour is drawn from  $\{a_{1,i}, a_{2,i},\dots, a_{|O_i|,i}\}$.
	Since $n_i = R_{\Gamma|_{O_i}}(a_{1,i}, a_{2,i},\dots, a_{|O_i|,i})$, $H$ is $\Gamma$-switch equivalent to a complete $m$-edge-coloured graph that contains, for some $1 \leq j \leq |O_i|$, a monochromatic complete graph on $a_{j,i}$ vertices of colour $c_{j,i}$.
	Therefore $R_\Gamma(a_1,a_2,\dots, a_m) \leq n = R(n_1,n_2,\dots,n_t)$.
\end{proof}

\subsection{Abelian Groups}

When $\Gamma$ is abelian the order of the switches in a switching sequence $\Pi$ does not matter.
And so we may assume that all of the switches at vertex $u_1$ are done before those at $u_2$ and so on.
Further, since:
\[
G^{(u_1,\pi)(u_1, \pi^\prime)} = G^{(u_1, \pi\circ\pi^\prime)}.
\]
we may assume there will be a single switch at each vertex.
Therefore every switching sequence has a natural canonical form:
\[
\Pi = (u_1,\pi_1),(u_2,\pi_2),\dots,(u_n,\pi_n).
\]
The set of switching sequences forms a group isomorphic to $\Gamma^n$, which acts on $\mathcal{K}_n^m$. 
The orbits of this group action are exactly the equivalence classes of the relation $\sim$.
In \cite{B25} and \cite{CT2004} and the authors consider the group structure where $\Gamma$ is not abelian.

When $\Gamma$ is abelian, we  associate a corresponding $m$-edge-coloured graph $P(G)$ that encodes the orbit containing $G$.
In \cite{S14}, and elsewhere, this construction is described for $2$-edge-coloured graphs and switching with respect to $S_2$.
This approach is extended to $m$-edge-coloured graphs and cyclic groups and then to arbitrary abelian groups \(\Gamma \leq S_m\) respectively in \cite{B09} and \cite{L21}.
Below we describe this construction when $G$ is a complete $m$-edge-coloured graph.

Let $G$ be a $m$-edge-coloured complete  graph.
Let $V(G) = \{v_1,v_2,\dots, v_n\}$.
We form $P(G)$ from $G$ by taking $|\Gamma|$ copies of $G$, indexed with elements of $\Gamma = \{\gamma_1=e,\gamma_2,\dots,\gamma_{|\Gamma|}\}$.
For $1 \leq i \leq |\Gamma|$, let $V(G_{\gamma_i}) = \{v_{1,\gamma_i},v_{2,\gamma_i}, \dots, v_{n,\gamma_i} \}$.
We complete the construction of $P(G)$ by adding edges between copies of $G$ such that for all $1 \leq k < \ell \leq n$ and all $1 \leq i < j \leq |\Gamma|$, if $c(v_k v_\ell) = t$, then  $v_{k,\gamma_i}v_{\ell,\gamma_j} \in P(G)$ and $c(v_{k,\gamma_i}v_{\ell,\gamma_j}) = \gamma_i \circ \gamma_j(t)$.
The colour of edge $v_{k,\gamma_i}v_{\ell,\gamma_j}$ is the colour that would result in $\Gamma$-switching in $G$ with $\gamma_i$ at  $v_k$ and then with $\gamma_j$ at $v_\ell$.
Corresponding vertices in copies of $G$ in $P(G)$ are not adjacent.
And so $P(G)$ is not a complete $m$-edge-coloured graph.

The structure of the graph $P(G)$ for arbitrary $G$ is studied in \cite{B09}.
The following result follows similarly to Theorem 8 therein. 
\begin{theorem}\label{thm:P(G)}
	Let $\Gamma \leq S_m$ be an abelian group.
	Let $G$ be a complete $m$-edge-coloured graph on $n$ vertices.
    Let $n^\prime \leq n$.
    The $m$-edge-coloured graph $P(G)$ contains a monochromatic complete graph on $n^\prime$ vertices as a subgraph if and only if there is an element of $[G]$ that has a monochromatic complete graph on $n^\prime$ vertices as a subgraph.
\end{theorem}

By definition, a $\Gamma$-switch Ramsey number is bounded above by the corresponding classical Ramsey number.
When $\Gamma$ abelian and its action on $[m]$ is semi-regular we can improve this bound by a factor of $|\Gamma|$ using $P(G)$.

\begin{theorem}\label{thm:pushGraph}
	If $\Gamma \leq S_m$ is abelian and the action of $\Gamma$ on $[m]$ is semi-regular, then 
	    \[R_\Gamma(a_1,a_2,\dots, a_m) \leq \left\lceil \frac{1}{|\Gamma|} R(a_1,a_2,\dots, a_m)\right\rceil.\]
\end{theorem}
\begin{proof}

	Let $G$ be a $m$-edge-coloured complete graph on $n=\lceil \frac{1}{|\Gamma|} R(a_1,a_2,\dots a_m)\rceil$ vertices.
	
	By construction, $P(G)$ has at least $R(a_1,a_2,\dots, a_m)$ vertices.
	We construct $P^\star(G)$  from $P(G)$ by adding an edge between corresponding vertices in the $|\Gamma|$ copies of $G$ used to construct $P(G)$.
	Since $R(3,3,\dots,3) > m \geq |\Gamma|$, this can be done in such a way that the subgraph induced by corresponding vertices of $G$ contains no monochromatic complete graph.
	That is for all $1 \leq k \leq m$, $G[\{v_{k,\gamma_1}, v_{k,\gamma_2},\dots, v_{k,\gamma_m}\}]$ contains no monochromatic complete graph.
    Thus any monochromatic complete graph in $P^\star(G)$ contains at most one edge whose ends correspond to the same vertex in $G$.
    We prove such a monochromatic complete graph does not exist.
    
	Since $P^\star(G)$ has at least $R(a_1,a_2,\dots, a_m)$ vertices, there exists $1 \leq t \leq m$ such that $P^\star(G)$ contains a monochromatic complete graph, $H$, of colour $t$ and order $a_t$.
	If $H$ contains no edge of the form $v_{k,\gamma_i}v_{k,\gamma_j}$, then $H$ is a subgraph of $P(G)$.
	And so by Theorem \ref{thm:P(G)}, $G$ can be $\Gamma$-switched to obtain a copy of $H$ and the result follows.
	
	Otherwise, assume $H$ contains an edge of the form $v_{k,\gamma_i}v_{k,\gamma_j}$.
	Without loss of generality, assume the following:  $H$ contains the edge $v_{1,\gamma_1}v_{1,\gamma_2}$ and $c_{P^\star(G)}(v_{1,\gamma_1}v_{1,\gamma_2}) = 1$.
	Let $x = v_{1,\gamma_1}$ and $y = v_{1,\gamma_2}$.
	Since $H$ has at least three vertices, then it contains a vertex other than $x$ and $y$.
	Let $w \neq x,y$ be a vertex of $H$.
	By construction there exists $2 \leq \ell \leq n$ and $1 \leq j \leq |\Gamma|$ such that  $w= v_{\ell,\gamma_j}$.
	Since $H$ is monochromatic, $c(xy)= c(xw) = c(yw)$.
	
	Without loss of generality assume $\ell = 2$ and $j \in \{1,3\}$.
	See Figure \ref{fig:P(G)proof} for corresponding configurations for $j =1,3$.
	The case $j = 1$ is symmetric to that of $j = 2$.

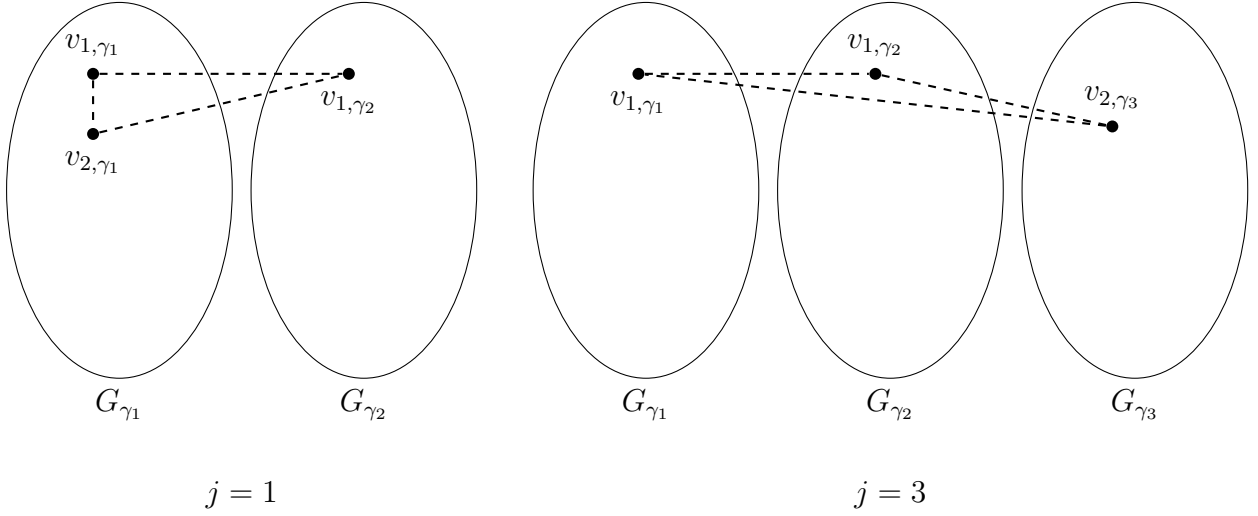
\begin{figure}
    \centering
\begin{tikzpicture}[vtx/.style={circle, fill=black, inner sep=1.6pt}]

\begin{scope}[shift={(0,0)}]
  \draw (0,0) ellipse [x radius=1.5, y radius=2.5];
  \draw (3.25,0) ellipse [x radius=1.5, y radius=2.5];

  \node[vtx, label = $v_{1,\gamma_1}$] (a1) at (-0.35, 1.55) {};
  \node[vtx, label = below:$v_{2,\gamma_1}$] (a2) at (-0.35, 0.75) {};
  \node[vtx, label = below:$v_{1,\gamma_2}$] (b1) at (3.05, 1.55) {};

  \draw[thick,dashed] (a1) -- (a2);
  \draw[thick, dashed] (a1) -- (b1);
  \draw[thick, dashed] (a2) -- (b1);

  \node at (0,-2.85) {$G_{\gamma_1}$};
  \node at (3.25,-2.85) {$G_{\gamma_2}$};

  \node at (1.625,-4.05) {$j=1$};
\end{scope}

\begin{scope}[shift={(7.0,0)}]
  \draw (0,0) ellipse [x radius=1.5, y radius=2.5];
  \draw (3.25,0) ellipse [x radius=1.5, y radius=2.5];
  \draw (6.50,0) ellipse [x radius=1.5, y radius=2.5];

  \node[vtx, label = below:$v_{1,\gamma_1}$] (c1) at (-0.10, 1.55) {};
  \node[vtx, label = $v_{1,\gamma_2}$] (c2) at (3.05, 1.55) {};
  \node[vtx, label = $v_{2,\gamma_3}$]  (c3) at (6.20, 0.85) {};
  
  \draw[thick, dashed] (c1) -- (c2);
  \draw[thick, dashed] (c2) -- (c3);
  \draw[thick, dashed] (c3) -- (c1);

  \node at (0,-2.85) {$G_{\gamma_1}$};
  \node at (3.25,-2.85) {$G_{\gamma_2}$};
  \node at (6.50,-2.85) {$G_{\gamma_3}$};

  \node at (3.25,-4.05) {$j=3$};
\end{scope}

\end{tikzpicture}  
    \caption{Configurations for $j = 1,3$ in the proof of Theorem \ref{thm:P(G)}}
    \label{fig:P(G)proof}
\end{figure}

	Consider first the case $j = 1$.
	In this case, $w$ is contained in the same copy of $G$ has $u$.
	Therefore $c_G(xw) = c_{P(G)}(xw) =  c_{P(G)}(yw)$.
	And so
	\[
	1 = c_G(xw) = c_{P(G)}(xw) = e\circ \gamma_2(1) = \gamma_2(1)
	\]
	Therefore $\gamma_2$ fixes $1$, a contradiction as no non-trivial element of $\Gamma$ fixes an element of $[m]$ since $\Gamma$ is semi-regular.
	
	Consider now the case $j = 3$.
	In this case,
	\[
	1 = c_{P(G)}(xw) = c_{P(G)}(yw)
	\]
	Therefore
	\[
	t = \gamma_3(1) = \gamma_2\circ\gamma_3(1),
	\]
	which again implies the existence of a non-trivial element of $\Gamma$ that fixes an element of $\{1,2,\dots, m\}$, or implies $\gamma_2 = \gamma_3=e$.
    The former contradicts the semi-regularity of $\Gamma$.
    The latter contradicts $j = 3$.
	
	Therefore $H$ contains no edge of the form $v_{k,\gamma_i}v_{k,\gamma_j}$.
    By Theorem \ref{thm:P(G)}, $G$ can be $\Gamma$-switched to contain, for some $1 \leq i \leq m$ a complete monochromatic subgraph of colour $i$ and order $a_i$.
    Therefore $R(a_1,a_2,\dots, a_m) \leq \left\lceil \frac{1}{|\Gamma|} R_\Gamma(a_1,a_2,\dots a_m)\right\rceil$.
\end{proof}

Using $P(G)$ we also obtain the following.

\begin{theorem}
	If $\Gamma$ is abelian, then $R(a_1, a_2 \dots, a_m) \leq R_\Gamma(|\Gamma|a_1,|\Gamma|a_2,\dots, |\Gamma|a_m)$.
\end{theorem}

\begin{proof}
	Let $n = R_\Gamma(|\Gamma|a_1, |\Gamma|a_2, \dots, |\Gamma|a_m)$ and let $G$ be a complete $m$-edge-coloured graph on $n$ vertices.
	By Theorem \ref{thm:P(G)}, for some $1 \leq i \leq m$, $P(G)$ contains a complete monochromatic subgraph $H$ of order $|\Gamma|a_i$ and colour $i$.
	By the pigeonhole principle, there exists $1 \leq j \leq |\Gamma|$ such that $G_{\gamma_j}$ contains at least $a_i$ of the vertices of $H$.
    Since $G_{\gamma_j}$ is a copy of $G$, $G$ contains a monochromatic complete graph of colour $i$ and order $a_i$.
	Therefore $R(a_1, a_2, \dots, a_m) \leq n$.	
\end{proof}

\begin{corollary}\label{cor:doubleBound}
For all $a,b \geq 2$,
        $$R\left(a, b\right) \leq R_{S_2}(2a,2b) \leq \left\lceil\frac{1}{2}R(2a,2b)\right\rceil.$$
\end{corollary}

\subsection{Commutator Subgroups}
When $\Gamma$ is sufficiently rich we can, through a sequence of switches, change the colour of a single edge of a complete $m$-edge-coloured graph without affecting the colours on any other edges of the graph.
Let $G$ be a complete $m$-edge-coloured graph and let $\Gamma$ be a  subgroup of $S_m$.
Let $\alpha,\beta \in \Gamma$.
Consider the commutator $\gamma = \beta^{-1}\alpha^{-1}\beta\alpha$.
If $\gamma$ is not the identity, then there exists $i \neq j$ such that $\gamma(i) = j$.
Assume $e = uv$ has colour $i$ in $G$, and consider the following $\Gamma$-switching sequence: $(u,\alpha),(v,\beta), (u,\alpha^{-1}), (v,\beta^{-1})$. 
Under this sequence, edge $e$ has changed from colour $i$ to colour $j$ and no other edge has changed its colour.

From this observation, we see that the action of $[\Gamma,\Gamma]$ on $\{1,2,3,\dots, m\}$ partitions the set of colours into orbits such that if $i$ and $j$ are in the same orbit, then there is a sequence of $\Gamma$-switches that change a single edge in $G$ from $i$ to $j$, leaving all other edges unchanged. 
And so, if a complete $m$-edge-coloured graph contains a subgraph of order $t$ where the colour of all the edges are drawn from the same orbit under the action of $[\Gamma,\Gamma]$, then $H$ can be $\Gamma$-switched to be monochromatic of any colour within this orbit.

\begin{lemma}\label{lem:commSubgroup}
Let $H$ be a complete $m$-edge-coloured graph. 
Let $\Gamma$ be a subgroup of $S_m$.
If the colours of all of the edges of $H$ are drawn from the same orbit $O$ induced by  the action of $[\Gamma,\Gamma]$ on $[m]$, then for every $j \in O$, $H$ is $\Gamma$-switch equivalent to a monochromatic complete graph of colour $j$. 
\end{lemma}

\begin{proof}
Let $\Gamma \leq S_m$.
Let $\mathcal{O}$ be the set of orbits induced by the action of $[\Gamma,\Gamma]$ on $[m]$.
Let $O \in \mathcal{O}$.
Let $H$ be a complete $m$-edge-coloured graph such that all of the edge colours are contained in $O$.
Let $j \in O$.
Let $e =uv$ be an edge of $H$ of colour $k \neq j$.

Since $j,k \in O$, there exists a product of commutators $\gamma = \Pi_{i=1}^{\ell} \gamma_i = \Pi_{i=1}^{\ell}\beta_i^{-1}\alpha_i^{-1}\beta_i\alpha_i$ such that $\gamma(k) = j$.
Applying the $\Gamma$-switch sequence:
\[(u,\alpha_{\ell}),(v,\beta_\ell), (u,\alpha_\ell^{-1}), (v,\beta_\ell^{-1}),
(u,\alpha_{\ell-1}),(v,\beta_{\ell-1}), (u,\alpha_{\ell-1}^{-1}), (v,\beta_{\ell-1}^{-1})\dots\] 
\[\dots (u,\alpha_{2}),(v,\beta_2), (u,\alpha_2^{-1}), (v,\beta_2^{-1})(u,\alpha_{1}),(v,\beta_1), (u,\alpha_1^{-1}), (v,\beta_1^{-1})\]
changes the colour of $uv$ to $j$ while leaving all other edge colours unchanged.
 
Therefore $H$ can be $\Gamma$-switched to be monochromatic of colour $j$.
\end{proof}

\begin{theorem}\label{thm:commClassic}
    Let $G$ be a complete $m$-edge-coloured graph.
    Let $\Gamma$ be a subgroup of $S_m$.
    Let $\mathcal{O} = \{O_1,O_2,\dots, O_t\}$ be the orbits induced by the action of $[\Gamma,\Gamma]$ on $[m]$.
    For $n_i = \min\{a_j\mid  j \in O_i\}$,
\[
R_\Gamma(a_1,a_2,\dots,a_m) \leq R(n_1,n_2,\dots, n_t).
\]
\end{theorem}

\begin{proof}
Let $n = R(n_1,n_2,\dots, n_t)$.
Consider $G$, a complete $m$-edge-coloured graph on $n$ vertices.
Let $\mathcal{O} = \{O_1,O_2,\dots, O_t\}$ be the orbits induced by the action of $[\Gamma,\Gamma]$ on $[m]$.
By relabelling each edge of colour $i$ with the orbit containing $i$, we arrive at a $t$-edge-coloured graph.
By choice of $n$, $G$ contains, for some $1 \leq i \leq t$, a complete monochromatic graph $H$ of order $n_i = \min\{a_j\mid  j \in O_i\}$.
The result follows from Lemma \ref{lem:commSubgroup}, as we may $\Gamma$-switch so that each edge in $H$ has the same colour.
\end{proof}

\begin{corollary}\label{cor:Commtrans}
    Let $G$ be a complete $m$-edge-coloured graph.
    Let $\Gamma$ be a subgroup of $S_m$.
    If $[\Gamma,\Gamma]$ acts transitively on $[m]$, then 
    \[
        R_\Gamma(a_1,a_2,\dots,a_m) = \min\{a_i\mid  1 \leq i \leq m\}.
    \]
\end{corollary}

For $\Gamma = A_m$, $[\Gamma,\Gamma]$ acts transitively on $[m]$.

\begin{theorem}
    Let $G$ be a complete $m$-edge-coloured graph.
    Then
    \[R_{A_m}(a_1,a_2,\dots,a_m) = \min\{a_i\mid  1 \leq i \leq m\}.\]
\end{theorem}

Recall the quotient group $\Gamma/[\Gamma,\Gamma]$ is abelian and acts on the set of orbits induced by the action of $[\Gamma,\Gamma]$ on $[m]$.

\begin{lemma}\label{lem:Quotient}
 Let $G$ be a complete $m$-edge-coloured graph.
    Let $\Gamma$ be a subgroup of $S_m$.
    Let $\mathcal{O} = \{O_1,O_2,\dots, O_t\}$ be the orbits induced by the action of $[\Gamma,\Gamma]$ on $[m]$.
    For $n_i = \min\{a_j \mid j \in O_i\}$,
\[R_{\Gamma}(a_1,a_2,\dots,a_m) \leq R_{\Gamma/[\Gamma,\Gamma]}(n_1,n_2,\dots,n_t).\]
\end{lemma}

\begin{proof}
Let $n = R_{\Gamma/[\Gamma,\Gamma]}(n_1,n_2,\dots,n_t)$.
And consider, $G$, a complete $m$-edge-coloured graph on $n$ vertices.
Let $\mathcal{O} = \{O_1,O_2,\dots, O_t\}$ be the orbits induced by the action of $[\Gamma,\Gamma]$ on $[m]$.

By relabelling each edge with the orbit containing the colour of the edge, we arrive at a $t$-edge-coloured graph.
By choice of $n$, $G$ contains, for some $1 \leq i \leq t$, a subgraph $H$ which can be $\Gamma/[\Gamma,\Gamma]$-switched to obtain a complete graph $H$ where the orbit labels are all $O_i$.
By definition this switching can be done using a representative elements of $\Gamma$.
By choice of $n$, $G$ contains, for some $1 \leq i \leq t$, $H$, a subgraph $H$ which can be $\Gamma$-switched to obtain a complete graph $H$ where the orbit labels are all $O_i$.
The result follows from Lemma \ref{lem:commSubgroup}.
    \end{proof}

For $\Gamma = D_m$, $[\Gamma,\Gamma]$ has either a single orbit or two orbits, depending on the parity of $m$.
Combining the results of Corollary \ref{cor:Commtrans} and Lemma \ref{lem:Quotient} yields the following result.
    
\begin{theorem}
    Let $G$ be a complete $m$-edge-coloured graph and let $\Gamma = D_m$.
    If $m$ is odd, then 
    \[R_{\Gamma}(a_1,a_2,\dots,a_m) = \min\{a_i\mid  1 \leq i \leq m\}.\]
    Otherwise, if $m$ is even, then
    \[R_{\Gamma}(a_1,a_2,\dots,a_m) \leq R_{S_2}(n_1,n_2)\]
    where $n_1 = \min\{a_1, a_3, \dots, a_{m-1}\}$ and $n_2 = \min\{a_2,a_4,\dots, a_{m}\}$.
\end{theorem}

By applying Theorem \ref{thm:inTrans} to $R_{\Gamma/[\Gamma,\Gamma]}(m_1,m_2,\dots,m_k)$ we obtain the following:
\begin{theorem}\label{thm:comIntrans}
Let $G$ be a complete $m$-edge-coloured graph.
Let $\Gamma$ be a subgroup of $S_m$.
Let $\mathcal{O} = \{O_1,O_2,\dots, O_k\}$ be the orbits induced by the action of $[\Gamma,\Gamma]$ on $[m]$.
Let $\mathcal{P} = \{P_1,P_2,\dots, P_q\}$ be the orbits induced by the action of $\Gamma/[\Gamma,\Gamma]$ on $\mathcal{O}$.
For $n_i = \min\{a_j \mid j \in O_i\}$ and
$m_j = R_{\Gamma/[\Gamma,\Gamma]}(n_{c_{1,j}}, n_{c_{2,j}},\dots n_{c_{|P_j|,j}})$,
$$R_{\Gamma}(a_1,a_2,\dots,a_t) \leq R(m_1,m_2,\dots,m_q).$$
\end{theorem}

\section{Groups Acting Transitively on the Set of Colours}\label{sec:trans}

Let $G$ be a complete graph and let $v \in V(G)$. 
We say $G$ is \emph{homogenised at $v$} when all of the edges incident with $v$ have the same colour.
We use the notation $G_v$ to refer to a graph that is homogenised at $v$.
When $G$ is homogenised at $v$, then switching at $v$ will preserve this property.
When $\Gamma$ acts transitively on $[m]$ there is always a switching sequence where the switch at $v$ is trivial and the resulting graph is homogenised at $v$.
Subsequently we may switch at $v$ so that $G$ is homogenised at $v$ with respect to any colour.
We use $G_v^i$ to refer to the homogenisation at $v$ where all of the edges incident with $v$ have colour $i$.
Using these ideas we prove the following bound, which generalises a corresponding approach in \cite{M24} for $m = 2$.

\begin{theorem}\label{thm:plusOne}
	If $\Gamma \leq S_m$ acts transitively on the set of colours, then
	\[R_\Gamma(a_1, a_2,\dots, a_m) \leq R(a_1-1,a_2-1,\dots, a_m-1)+1.\]
\end{theorem}

\begin{proof}
	Let $\Gamma \leq S_m$ act transitively on the set of colours.
	Let $n= R(a_1-1,a_2-1,\dots a_m-1)+1$ and consider an arbitrary complete $m$-edge-coloured graph $G$ on $n$ vertices.
	Since $\Gamma$ acts transitively, we may assume $G$ is homogenised at a vertex $v \in V(G)$; we may switch at each vertex other than $v$ so that all edges incident with $v$ have the same colour.
	The subgraph induced by $V(G) \setminus \{v\}$ has $R(a_1-1,a_2-1,\dots, a_m-1)$ vertices.
	And so, for some $1 \leq i \leq m$, $G-v$ contains a monochromatic complete graph on $a_i-1$ vertices of colour $i$.
	Since $\Gamma$ acts transitively on the set of colours we may switch at $v$ so that all of the edges incident with $v$ have colour $i$, yielding a monochomatic complete graph on $a_i$ vertices of colour $i$.
\end{proof}

In the proof of Theorem \ref{thm:plusOne}, we find a monochromatic complete subgraph containing a fixed vertex $v$ by first homogensing at $v$ and then by finding (without switching) a monochromatic complete subgraph in $G_v-v$.
This approach will be useful in our study of $C_m$-Ramsey numbers. 

 \begin{theorem}\label{thm:smallerN}
     Let $m,n, n^\prime \geq 2$ with $n^\prime \leq n$.
     Let $\Gamma \leq S_m$ act transitively on the set of colours.
     Let $G \in \mathcal{K}_n^m$ and $v \in V(G)$.
     There exists $H \in [G]$ in which $v$ is contained in a monochromatic $K_{n^\prime}$ if and only if $G^{(1)}_v-v$ contains a monochromatic $K_{n^\prime-1}$.
 \end{theorem}

\begin{proof}
	Let $m,n, n^\prime \geq 2$ with $n^\prime \leq n$.
	Let $\Gamma \leq S_m$ act transitively on the set of colours.
	Let $G \in \mathcal{K}_n^m$ and $v \in V(G)$.

	Assume there exists $H \in [G]$ in which $v$ is contained in a monochromatic $K_{n^\prime}$.
	Without loss of generality let $v = u_n$, and assume this complete subgraph has edges of colour $1$ and has vertex set $\{u_1,u_2,\dots, u_{n^\prime-1}, u_n\}$.
	Therefore there exists a switching sequence 
	$$\Pi = (u_1,\pi_1)(u_2,\pi_2),\dots,(u_{{n^\prime}-1},\pi_{{n^\prime}-1}),(u_{n^\prime},e),(u_{n^\prime+1},e),\dots, (u_{n-1},e),(u_{n},\pi_n) $$
	 such that $v$ is contained in a monochromatic subgraph of order $n^\prime$ in $G^\Pi$.
	For $n^\prime \leq j \leq n-1$ let $\gamma_i \in \Pi$ such that $\gamma_i(c_G(u_ju_n)) = \phi_1(c_G(u_1u_n))$.
	For $1 \leq i \leq n$ let
	\[
	\phi_i = 
	\begin{cases}
		\pi_i, & i \leq n^\prime-1 \text{ or } i = n\\
		\gamma_i, & n^\prime  \leq i \leq n-1.
	\end{cases}
	\]
	Let $\Phi = (u_1,\phi_1),(u_2,\phi_2),\dots,(u_n,\phi_n)$.
	By construction, $G^\Phi$ is homogenised at $v$ and $$G^\Phi[u_1,u_2,\dots u_{n^{\prime}-1}] \cong K^{(c_{G^\Phi}(u_1u_n))}_{n^\prime-1}.$$
	
	Assume  $G^{(1)}_v-v$ contains a monochromatic $K_{n^\prime-1}$.
	Since $\Gamma$ acts transitively on $[m]$ we may switch at $v$ so that all of the edges incident with $v$ have the same colour as the edges in this monochromatic $K_{n^\prime-1}$.
	Therefore there exists $H \in [G]$ in which $v$ is contained in a monochromatic $K_{n^\prime}$.
\end{proof}

In general we do not expect the bound in Theorem \ref{thm:plusOne} to be tight. 
However, Mutar, Sivaraman and Slilaty prove this bound is tight for $R_{S_2}(3,t)$ and $R_{S_2}(4,4)$ \cite{M24}.
In what follows we prove the bound is tight for $R_{C_3}(4,4,4)$ and $R_{C_4}(3,4,3,4)$.
We also provide evidence that suggests it may be tight for $R_{C_4}(4,4,4,4)$.
This latter result provides a possible new approach to improving the long-standing lower bound on $R(3,3,3,3)$.

\subsection{Cyclic Groups}
As noted in the previous section, When $\Gamma$ is abelian, every switching sequence can be expressed as 
\[
\Pi = (u_1,\pi_1),(u_2,\pi_2),\dots,(u_n,\pi_n).
\]
The set of switching sequences forms a group isomorphic to $\Gamma^n$, which acts on the set of $m$-edge-coloured complete graphs.

When $\Gamma = C_m$, we may express each element of $\Gamma$ as $(123\cdots m)^k$ for some $k \in [m]$.
In this case we can compute the colour of every edge of $G^\Pi$ as follows
\[c_{G^\Pi}(u_iu_j) \equiv c_{G}(u_iu_j) + k_i + k_j \pmod m,
\]
where $\pi_i = (123\cdots m)^{k_i}$ and $\pi_j = (123\cdots m)^{k_j}$.

Combining together these ideas permits us to count the number of orbits of $\mathcal{K}_n^m/C_m^n$, which leads to bounds on the $C_m$-switch  Ramsey number.

\begin{theorem}\label{thm:ker}
    Let $m,n \geq 1$ and $\Gamma = C_m$.
   	If $m$ is odd, then the  kernel of the action of $\Gamma^n$ on $\mathcal{K}_n^m$ is $\{e\}$.
    If $m$ is even, then the kernel of the action of $\Gamma^n$ on $\mathcal{K}_n^m$ is  $\{e, ((12\cdots m)^{m/2},(12\cdots m)^{m/2},\dots, (12\cdots m)^{m/2})\}$.
\end{theorem}

\begin{proof}
    Consider $\Pi \in \Gamma^n$ such that $H^\Pi = H$ for all $H \in \mathcal{K}_n^m$.
    Consider $u_iu_j \in E(H)$.
    Since $\Gamma = C_m$ there exists $k_i,k_j \in [m]$ such that $\pi_i = (12\cdots m)^{k_i}$ and $\pi_j = (12\cdots m)^{k_j}$.
    And so 
    $$c_{H^\Pi}(u_iu_j) \equiv c_{H}(u_iu_j) + k_i + k_j \pmod m.$$
    If $H^\Pi = H$, then for all $i,j \in [n]$,
    \[
     k_i + k_j  \equiv 0 \pmod m.
    \]
    When $m$ is even the only solutions are $k_\ell = 0$ for all $\ell \in [n]$ or $k_\ell = m/2$ for all $\ell \in [n]$.
    When $m$ is odd the only solution is $k_\ell = 0$ for all $\ell \in [n]$. 
\end{proof}

An elementary application of the Orbit Stabiliser Theorem yields the following.

\begin{corollary}\label{cor:countClass}
 Let $m,n \geq 1$ and  $G \in \mathcal{K}_n^m$.
 
\begin{itemize}
 \item If  $m$ is even, then  $|Orb_{C_m^n}(G)|= \frac{m^n}{2}$ and $|\mathcal{K}_n^m/C_m^n| = 2m^{{n \binom{n}{2}} - n}$.
     \item If $m$ is odd, then $|Orb_{C_m^n}(G)| = m^n$ and $|\mathcal{K}_n^m/C_m^n| =m^{\binom{n}{2} - n}$.
\end{itemize}
\end{corollary}

\begin{theorem}\label{thm:allThree}
    For all $m \geq 2$, $R_{C_m}(3,3,\dots, 3) = 3$.
\end{theorem}

\begin{proof}
If $m$ is odd, then $|\mathcal{K}_3^m/{C_m}^n| = m^{{\binom{3}{2}} - 3} =  1$, which implies that any pair of elements of $\mathcal{K}_3^m$ are in the same orbit. 
In other words, any $m$-edge-coloured complete graph on $3$ vertices is $C_m$-switch equivalent to $K_3^{(1)}$.

If $m$ is even, then $|\mathcal{K}_3^m/{C_m}^n| = 2m^{{\binom{3}{2}} - 3} =  2$. 
To complete the proof it suffices to show that $K_3^{(1)}$ and $K_3^{(2)}$ are in different orbits of $\mathcal{K}_3^m/C_m^n$.
This then implies that every element of $\mathcal{K}_3^m$ is $C_m$-switch equivalent to one of  $K_3^{(1)}$ or $K_3^{(2)}$.

When $m$ is even, a switch at a vertex $u$ with an element of $C_m$ preserves the parity of the sum of edges incident with $u$ and thus the parity of the sum of the edges in the $m$-edge-coloured graph.
Since the parity of the sum of edges in $K_3^{(1)}$ and in $K_3^{(2)}$ are different, these $m$-edge-coloured graphs must lay in different orbits of $\mathcal{K}_3^m/C_m^n$.
Therefore every $m$-edge-coloured complete graph on $3$ vertices is $C_m$-switch equivalent to $K_3^{(1)}$ or $K_3^{(2)}$.
\end{proof}

\begin{lemma}\label{lem: even uniform switching}
    Let $m$ be even and $n \geq 2$. 
    The monochromatic complete graphs $K_{n}^{(i)}$ and $K_{n}^{(j)}$ are $C_m$-switch equivalent if and only if $j-i\equiv 0\pmod 2$.
\end{lemma}

\begin{proof}
    As in our approach in the proof of Theorem \ref{thm:ker}, it suffices to find a solution to the following system of congruences.
    \[
    i+k_{t_1}+k_{t_{2}}\equiv~j\pmod m \quad 1 \leq t_1 < t_2 \leq n
    \]
    When $j-i\equiv 0\pmod 2$ we have the unique solution $k_1\equiv\cdots\equiv k_n\equiv\frac{j-i}{2}\mod m$.
    Otherwise, when $j-i\equiv 1\pmod 2$ there is no solution.
\end{proof}

\begin{theorem}\label{thm:evenOERamsey}
    If $m$ is an even integer, then
    \begin{equation*}
        R_{C_m}(a_1,a_2,\ldots,a_m)=R_{C_m}(\mathcal{O},\mathcal{E},\mathcal{O},\mathcal{E},\ldots,\mathcal{O},\mathcal{E})
    \end{equation*}
    where $\mathcal{O}=\min\{a_1,a_3,\ldots,a_{m-1}\}$ and $\;\mathcal{E}=\min\{a_2,a_4,\ldots,a_m\}$.
\end{theorem}
\begin{proof}
    Clearly $R_{C_m}(a_1,a_2,\ldots,a_m)\geq R_{C_m}(\mathcal{O},\mathcal{E},\mathcal{O},\mathcal{E},\ldots,\mathcal{O},\mathcal{E})$.

    Let $n=R_{C_m}(\mathcal{O},\mathcal{E},\mathcal{O},\mathcal{E},\ldots,\mathcal{O},\mathcal{E})$. Let $i$ be an odd integer and let  $j$ an even integer such that $a_i=\mathcal{O},a_j=\mathcal{E}$. 
    Any $m$-edge-colouring of $K_n$ can be $C_m$-switched to contain a $K_{\mathcal{O}}^{(x)}$ or $K_{\mathcal{E}}^{(y)}$ for $x\in\{1,3,\ldots,m-1\},~y\in\{2,4,\ldots,m\}$. 
    By Lemma \ref{lem: even uniform switching} for even $m$, $K_{\mathcal{O}}^{(x)}$ and $K_{\mathcal{O}}^{(i)}$ are $C_m$-switch equivalent, as are $K_{\mathcal{E}}^{(y)}$ and $K_{\mathcal{E}}^{(j)}$.
\end{proof}

\begin{corollary}
    For even $m$ and for all $a_3,\dots a_m \geq 3$, $R_{C_m}(3,3,a_3,\dots, a_m) = 3$.
\end{corollary}

When $m$ is odd we apply similar techniques to obtain the following.

\begin{lemma}\label{lem: odd uniform switching}
    If $m$ is odd, then $K_{n}^{(i)}$ and $K_{n}^{(j)}$ are $C_m$-switch equivalent for any $i,j\in[m]$.
\end{lemma}

\begin{theorem}\label{thm:oddRamsey}
    If $m$ is an odd integer, then
    \begin{equation*}
        R_{C_m}(a_1,a_2,\ldots,a_m)=R_{C_m}(N,N,\ldots,N)
    \end{equation*}
    where $N=\min\{a_1,a_2,\ldots,a_m\}$. 
\end{theorem}

\begin{corollary}
    For odd $m$ and for all $a_2,a_3,\dots, a_m \geq 3$, $R_{C_m}(3,a_2, a_3,\dots, a_m) = 3$.
\end{corollary}

Using the tools developed above, we study $R_{C_3}(4,a_1,a_2)$, $R_{C_4}(3,4,a_1,a_2)$ and  $R_{C_4}(4,4,a_1,a_2)$.
By Theorems \ref{thm:evenOERamsey} and \ref{thm:oddRamsey}, it suffices to study $R_{C_3}(4,4,4)$, $R_{C_4}(3,4,3,4)$ and  $R_{C_4}(4,4,4,4)$.
In the first two cases we find exact values for these parameters.

\begin{theorem}\label{thm:3434}
For $\Gamma = C_4$, $R_{\Gamma}(3,4,3,4) = R(2,3,2,3) + 1$.
\end{theorem}

\begin{proof}
By Theorem \ref{thm:plusOne},
$R_{\Gamma}(3,4,3,4) \leq R(2,3,2,3) + 1 = 7$.
Let $\mathcal{K}=\{K_3^{(1)},K_4^{(2)},K_3^{(3)},K_4^{(4)}\}$.
To complete the proof it suffices to find a $4$-edge-coloured graph on $6$ vertices that is not $C_4$-switch equivalent to one containing an element of $\mathcal{K}$.

Consider the $4$-edge-colouring of $K_5$ formed by partitioning $K_5$ into two disjoint $5$-cycles and respectively colouring them with colour $2$ and colour $4$.
Observe this $4$-edge-coloured graph contains no element of $\mathcal{K}$.
Let $G$ be the graph formed from this $4$-edge-coloured graph in  by adding a universal vertex, $v$, and colouring all edges incident with this edge with colour $1$.
By construction, $G$ contains no element of $\mathcal{K}$.
By Corollary \ref{thm:smallerN}, any element of $\mathcal{K}$ in an element of $[G]$ cannot contain $v$.
And so it suffices to prove $[G-v]$ contains no graph that has a element of $\mathcal{K}$.

Since every edge of $[G-v]$ has even colour the sum of colours in every copy of $K_3$ is even.
Following an argument similar to the one in the proof of Theorem \ref{thm:allThree} yields the conclusion that $G-v$ cannot be switched to contain $K_3^{(1)}$ or $K_3^{(3)}$.

Recall that a complete graph on $4$ vertices can be partitioned into two paths with $4$ vertices.
Every copy of $K_4$ in $[G-v]$ is isomorphic to one where one path on $4$ vertices has edges of colour $2$ and the other has colour $4$.

For such a copy of $K_4$ to be made monochromatic we must switch the edges of colour $2$ to colour $4$ or vice versa.
By Lemma \ref{lem: even uniform switching} it suffices to consider one of the cases.
Let $v_1v_2$, $v_2v_3$ and $v_3v_4$ be edges of colour $2$ in such a copy of $K_4$. 
Switching a copy of $K_4$ in $G-{u_1}$ to be monochromatic of colour $4$ requires solving the following system of simultaneous congruences:
    \begin{align*}
        k_1+k_2&\equiv 2 \pmod 4\\
        k_2+k_3&\equiv 2 \pmod 4\\
        k_3+k_4&\equiv 2 \pmod 4\\
        k_1+k_4&\equiv 0 \pmod 4\\
        k_1+k_3&\equiv 0 \pmod 4\\
        k_2+k_4&\equiv 0 \pmod 4.
    \end{align*}
One can verify no such solution exists.
Since no solution exists, no element of $[G-v]$ contains any member of $\mathcal{K}$,
which implies no element of $[G]$ contains any member of $\mathcal{K}$.
Therefore $R_\Gamma(3,4,3,4) > 6$, and so $R_{\Gamma}(3,4,3,4)=7$.
\end{proof}

Applying Theorem \ref{thm:evenOERamsey} now yields the following.

\begin{corollary}
For $\Gamma = C_4$ and all $a_1 \geq 3$ and $a_2 \geq 4$, $R_{\Gamma}(3,4,a_1,a_2)=7$.
\end{corollary}

Using a similar approach we compute $R_{C_3}(4,4,4)$. 
In \cite{G55} the authors give an example of a $3$-edge-coloured graph on $16$ vertices that contains no monochromatic $K_3$.

Given $R(3,3,3)= 17$ \cite{G55}, we proceed as in the proof Theorem \ref{thm:3434} where $G$ is constructed from the $16$ vertex graph by adding a new vertex $v$ that is adjacent to all other vertices with edges of colour $1$.
By computer search we confirm no element of $[G-v]$ contains a monochromatic $K_4$.

\begin{theorem}\label{thm:R418}
$R_{C_3}(4,4,4) = R(3,3,3)+1$.
\end{theorem}

\begin{proof}
	By Theorem \ref{thm:plusOne}, $R_{C_3}(4,4,4) \leq R(3,3,3)+1 = 18$.
	Let $G$ be the $3$-edge-coloured graph on $17$ vertices formed from the $3$-edge-coloured graph on $16$ vertices described in \cite{G55} by adding a universal vertex $v$, adjacent to all existing vertices with edges of colour $1$.
	By construction, $G$ is homogenised at $v$ and contains no monochromatic $K_4$.
	By computer search, $G-v$ cannot be switched to contain a monochromatic $K_4$.
	The result follows by Theorem \ref{thm:smallerN}.
\end{proof}

\begin{corollary}
For all $a_1, a_2 \geq 4$, $R_{C_3}(4,a_1,a_2)=18$.
\end{corollary}

Up to isomorphism, there are two 3-edge-colourings of $K_{16}$ that have no monochromatic $K_3$, the first of which is noted above and the second of which was first identified by an undergraduate student at the University of Waterloo in the 1960s (see note in  \cite{KS1968}).
Computer search on this latter 3-edge-colouring of $K_{16}$ with no monochromatic $K_3$ identified an element in its equivalence class that contains a monochromatic $K_4$.

In \cite{G55}, the author describes a $4$-edge-coloured graph on $41$ vertices that contains no monochromatic $K_4$.
Proceeding as above yields a $4$-edge-coloured complete graph $G$ on $42$ vertices that cannot be switched to contain a mononchromatic $K_4$ of any colour.
The best known upper bound for $R(3,3,3,3)$ is $62$ \cite{F04}.

\begin{theorem}
For $\Gamma = C_4$,  $43 \leq R_{\Gamma}(4,4,4,4) \leq R(3,3,3,3) +1$.
\end{theorem}

\begin{corollary}
For $\Gamma = C_4$ and all $a_1,a_2 \geq 4$,  $43 \leq R_{\Gamma}(4,a_1,4,a_2) \leq 63$.
\end{corollary}

The best known lower bound for $R(3,3,3,3)$ is $51$ \cite{C73}, which improved a previous bound of $50$ \cite{W71}. 
However, in both cases computer search shows that the monochromatic $K_3$-free $4$-edge-coloured graph used in the proof of the bound can be $C_4$-switched to obtain a monochromatic $K_4$.

\section{Conclusions and Future Directions}
In \cite{M24}, the authors prove  $R_{S_2}(4,4) = R(3,3) +1$.
Here we prove $R_{C_3}(4,4,4) = R(3,3,3) + 1$.
Our argument for $R_{C_4}(4,4,4,4) \geq43$ is based on a previous best known bound for $R(3,3,3,3)$.
And so it is natural to wonder whether $R_{C_m}(4,4,\dots, 4) = R(3,3,\dots,3)+1$ for any other values $m \geq 4$.
For fixed $m$, this result holds if and only if there exists an $m$-edge-coloured complete graph on $R(3,3,\dots,3)-1$ vertices containing no monochromatic copies of $K_3$ that also cannot be switched to contain a monochromatic copy of $K_4$.
Regardless, the result relating $R_{C_4}(4,4,4,4)$ and $R(3,3,3,3)$ provides a possible new approach to the long standing open problem of computing this latter value.
If indeed $R(3,3,3,3) > 51$, then it may be possible to construct a $4$-edge-colouring of a complete graph on $n>51$ vertices that does not switch to contain a monochromatic $K_4$.
This would imply $R(3,3,3,3) > n-1>51$.
Though such a graph may contain a monochromatic $K_3$, it would contain in its equivalence class a graph with a vertex deleted subgraph that contains no monochromatic $K_3$.

Recall the construction of $G$ in the proof of Theorem \ref{thm:R418} using the $16$ vertex graph described in \cite{G55}.
Since $G$ cannot be switched to contain a monochromatic copy of $K_4$, when we homogenise at any vertex $u$ with colour $1$ and then remove the homogenised vertex, the remaining graph must contain no monochromatic $K_3$.
Therefore  for every $u \in G$, $G_{u}-u$ is a graph on $16$ vertices that contains no monochromatic $K_3$.
And so it must be one of the two $3$-edge-coloured complete graphs on $16$ vertices that contain no monochromatic $K_3$.
In fact it turns out that for every $u \in V(G)$, $G_u - u$ is exactly the graph given in \cite{G55}.
This implies $G_u \cong G$ for all $u \in V(G)$, which raises the question: for which (complete) $m$-edge-coloured graphs $H$ does there exist a switching sequence $\Pi$ such that $G^\Pi \cong G$ (but $G^\Pi \neq G$)? 
In other words, for which graphs $G$ does there exist $H \in [G]$ such that $H \cong G$?

\bibliographystyle{abbrv}
\bibliography{references.bib}

\end{document}